\documentclass{article}
\usepackage[top=1in,bottom=1.1in,left=1in,right=1in]{geometry}

\usepackage[T1]{fontenc}
\usepackage{amsmath,amsthm,amsfonts,amssymb,graphicx}
\usepackage{hyperref,verbatim}

\theoremstyle{definition}
    \newtheorem{theorem}{Theorem}[section]
    \newtheorem{lemma}[theorem]{Lemma}
    \newtheorem{proposition}[theorem]{Proposition}
    \newtheorem{corollary}[theorem]{Corollary}
    \newtheorem{definition}[theorem]{Definition}
  \newtheorem{example}[theorem]{Example}
\theoremstyle{remark}
    \newtheorem{remark}[theorem]{Remark}
 \newcommand{\N}{\mathbb{N}}                                            
 \newcommand{\R}{\mathbb{R}}                                            
 \newcommand{\eps}{\varepsilon}                                         
 \newcommand{\nor}[1]{\left\Vert#1\right\Vert}                          
 \newcommand{\abs}[1]{\left\vert#1\right\vert}                          
 \newcommand{\comp}{{\small \,\circ\,}}                                 %
 
 \DeclareMathOperator{\fix}{Fix}                                        
 \newcommand{\m}[1]{\overline{#1}}    			                        
 \DeclareMathOperator{\CAT}{CAT}                                        
 \DeclareMathOperator{\diam}{diam}

\DeclareMathOperator*{\argmin}{argmin}

\begin{document}

\title{The Asymptotic Behavior of the Composition of Firmly Nonexpansive Mappings}
\author{David Ariza-Ruiz$^{a}$, Genaro L\'{o}pez-Acedo$^{a}$, Adriana Nicolae$^{b,c}$}
\date{}

\maketitle
\begin{center}
{\footnotesize
$^{a}$Departamento de An\'alisis Matem\'atico - IMUS, Universidad de Sevilla, Apdo. de Correos 1160, 41080 Sevilla, Spain
\ \\
$^{b}$Department of Mathematics, Babe\c s-Bolyai University, Kog\u alniceanu 1, 400084 Cluj-Napoca, Romania
\ \\
$^{c}$Simion Stoilow Institute of Mathematics of the Romanian Academy, Research group of the project PD-3-0152,\\ 
P.O. Box 1-764, RO-014700 Bucharest, Romania\\
\ \\
E-mail addresses:  dariza@us.es (D. Ariza-Ruiz), glopez@us.es (G. L\'{o}pez-Acedo), anicolae@math.ubbcluj.ro (A. Nicolae)
}
\end{center}

\begin{abstract}

In this paper we provide a unified treatment of some convex minimization problems, which allows for a better understanding and, in some cases, improvement of results in this direction proved recently in spaces of curvature bounded above. For this purpose, we analyze the asymptotic behavior of compositions of finitely many firmly nonexpansive mappings in the setting of $p$-uniformly convex geodesic spaces focusing on asymptotic regularity and convergence results.\\


\end{abstract}

\section{Introduction}\label{intro}
Various mathematical problems can be formulated as finding a common fixed point of finitely many firmly nonexpansive mappings.  These include, for instance, the convex feasibility problem~\cite{BauBor96}, which consists of finding a point in the intersection of a finite number of closed convex sets, or the problem of finding a common zero of a finite family of maximal monotone operators~\cite{Com04}.

Firmly nonexpansive mappings were introduced by Bruck~\cite{Bru73} in order to study sunny nonexpansive retractions onto closed subsets of Banach spaces. In Hilbert spaces, these mappings coincide with firmly contractive ones defined by Browder~\cite{Bro67}. A basic example of firmly nonexpansive mappings is the metric projection onto closed convex subsets of Hilbert spaces. This mapping or finite compositions thereof appear in different iterative methods used in the study of convex minimization problems. Results on the asymptotic behavior of some  such methods follow from analyzing compositions of finitely many firmly nonexpansive mappings, as it is the case of the cyclic projection algorithm employed in solving the convex feasibilty problem (see~\cite{BauBor96} for a deep treatment of this topic).

Another important example of firmly nonexpansive mappings is the resolvent of a monotone operator which was introduced by Minty~\cite{Min62}.  Since the subdifferential of a proper convex lower semi-continuous function is a maximal monotone operator, some splitting methods which involve resolvents and are used to solve convex minimization problems can be abstracted to compositions of firmly nonexpansive mappings. Using this approach, Bauschke, Combettes and Reich \cite{BauComRei05} (see also \cite{AckPre80}) proved, in particular, that having two proper convex  lower semicontinuous functions defined on a Hilbert space, one can apply alternatively the two resolvents  to obtain weak convergence to a solution (provided it exists) of a minimization problem associated to the two functions, see (\ref{def-fct-min}). It is noteworthy that the convex feasibilty problem for two sets can be reformulated as a problem of type (\ref{def-fct-min}) for both the consistent and the inconsistent case.

The most straightforward way to find a common fixed point or, more generally, a fixed point of the composition of finitely many firmly nonexpansive mappings is to iterate the mappings cyclically. The particular case of just one mapping was considered in Hilbert spaces by Browder~\cite{Bro67} who proved that, for any starting point, the Picard iterates of a firmly nonexpansive mapping converge weakly to a fixed point, whenever the fixed point set is nonempty. Although the class of firmly nonexpansive mappings is not closed under composition, in Hilbert spaces firmly nonexpansive mappings are averaged and the composition of averaged mappings is averaged. Thus, results for the general case of finitely many firmly nonexpansive mappings can be deduced from the asymptotic behavior of averaged mappings (see also \cite{Tse92} for a more general order of composing the mappings that still preserves weak convergence). In Banach spaces not every firmly nonexpansive mappings is averaged and so extensions of such results to this setting are not immediate. Bruck and Reich~\cite{BruRei77} introduced strongly nonexpansive mappings that generalize firmly nonexpansive ones when the space is uniformly convex. This class of mappings is closed under composition and so one can analyze compositions of firmly nonexpansive mappings using results on the asymptotic behavior of strongly nonexpansive mappings.

More recently, problems where the study of compositions of firmly nonexpansive mappings plays an important role (such as minimization problems \cite{Udr94} or abstract Cauchy problems\cite{May98}) have also been formulated in a nonlinear setting. Since our applications focus on minimizing convex functions in geodesic spaces, we mention some of the work in this direction. The proximal point method and modifications thereof for optimization problems in linear spaces have been extended to Riemanninan manifolds in~\cite{LiLopMar09,LiYao12} and references therein. Spaces of nonpositive curvature in the sense of Alexandrov (also known as CAT$(0)$ spaces) proved to be an appropriate setting for considering such problems too. Jost studied the concept of resolvents for convex functions to develop the theory of harmonic mappings~\cite{Jos97} which motivated Ba\v{c}\'{a}k \cite{Bac12} to extend the proximal point algorithm to this context. In \cite{Bac13}, a proximal splitting algorithm was also approached in this setting (assuming in addition local compactness) and applied to the computation of the geometric median and the Fr\'{e}chet mean of a finite set of points. Regarding the convex feasibility problem, results on the asymptotic behavior of the alternating projection method have been obtained in~\cite{BacSeaSim12,Nic13}. The minimization problem (\ref{def-fct-min}) was studied in this setting in \cite{Ban14}.

The goal of this work is to create a suitable theoretical framework so that previous results can be better understood and in some cases improved. To this end, we present a systematic study of the asymptotic behavior of compositions of a finite number of firmly nonexpansive mappings in $p$-uniformly convex geodesic spaces. Asymptotic regularity is a very important concept used in the study of the asymptotic behavior of sequences and was formally introduced by Browder and Petryshyn in~\cite{BroPet66}. In Section~\ref{section-as-reg} we give asymptotic regularity results for sequences obtained by applying a finite number of firmly nonexpansive mappings in a cyclic way. We assume first that they have a common fixed point, but consider also the case of two mappings which satisfy the weaker condition that their composition is not fixed point free. Moreover, we use techniques which originate from proof mining (see \cite{Koh08} for more details) to provide polynomial rates of asymptotic regularity, thus giving quantitative versions of such results and extending a corresponding result proved in \cite{Nic13}. Section \ref{section-conv} deals with $\Delta$-convergence results for these problems. Assuming appropriate compactness conditions we obtain strong convergence. When the mappings have a common fixed point we apply our findings to the cyclic projection method for finitely many sets in CAT$(\kappa)$ spaces improving previous results from \cite{BacSeaSim12,Bac13}. We show that the results concerning two firmly nonexpansive mappings for which the fixed point set of the composition is nonempty can be used to study the minimization problem (\ref{def-fct-min}) in CAT$(0)$ spaces which partially recovers the main result of \cite{Ban14} and yields a convergence result for the inconsistent convex feasibility problem for two sets.

%
 \section{Preliminaries}
%
\subsection{Some Notions on Geodesics Spaces} \label{subsection-geod-sp}
%
Let $(X,d)$ be a metric space. A {\it geodesic path} that joins two points $x,y \in X$ is a mapping $\gamma:[0,l]\subseteq\R\to X$ such that $\gamma(0)=x$, $\gamma(l)=y$ and $d(\gamma(t),\gamma(t'))=\abs{t-t'}$ for all $t,t'\in[0,l]$. The image $\gamma([0,l])$ of $\gamma$ is called a {\it geodesic segment} from $x$ to $y$.  We denote a point $z \in X$ belonging to such a geodesic segment by $z=(1-t)x\oplus ty$, where $t=d(z,x)/d(x,y)$. $(X,d)$ is a {\it (uniquely) geodesic space} iff every two points in $X$ are joined by a (unique) geodesic path. A subset $C$ of $X$ is {\it convex} iff it contains all geodesic segments that join any two points in $C$. For more details on geodesic metric spaces, see \cite{BriHae99}.

Ball, Carlen and Lieb introduced in~\cite{BalCarLie94} the notion of $p$-uniform convexity which plays an essential role in Banach space theory. Recall that a normed space $(X,\nor{\cdot})$ is said to be {\it $p$-uniformly convex} for $2\leq p<\infty$ iff there exists a constant $C\geq1$ such that for any $x,y,\in X$,
$$
\displaystyle \nor{\frac{x+y}{2}}^p\leq\frac{1}{2}\,\nor{x}^p+\frac{1}{2}\,\nor{y}^p-\frac{1}{C^p}\,\nor{\frac{x-y}{2}}^p.
$$
Recently, Naor and Silberman~\cite{NaoSil11} extended this notion to the setting of geodesic spaces in the following way.
\begin{definition}
Fix $1<p<\infty$. A metric space $(X,d)$ is called {\it $p$-uniformly convex} with parameter $c>0$ iff $(X,d)$ is a geodesic space and for any three points $x,y,z\in X$ and all $t\in[0,1]$,
\begin{equation}\label{eq:p-UC}
d(z,(1-t)x\oplus ty)^p\leq(1-t)\,d(z,x)^p+t\,d(z,y)^p-\frac{c}{2}\,t(1-t)\,d(x,y)^p.
\end{equation}
\end{definition}

Note that inequality~\eqref{eq:p-UC} guarantees that the space $X$ is uniquely geodesic (see~\cite[Lemma~2.2]{Oht07} for a proof in the case $p=2$). Also, any closed convex subset of a $p$-uniformly convex space is again a $p$-uniformly convex space with the same parameter.

It is well-known that any $L_p$ space over a measurable space is $p$-uniformly convex. In the setting of nonlinear spaces, every $\CAT(0)$ space is $2$-uniformly convex with parameter $c=2$. Actually, in this case (\ref{eq:p-UC}) characterizes $\CAT(0)$ spaces. For $\kappa>0$, any $\CAT(\kappa)$ space $X$ with $\diam(X)<\pi/(2\sqrt{\kappa})$ is a $2$-uniformly convex space with parameter $c=(\pi-2\sqrt{\kappa}\,\eps)\,\tan(\sqrt{\kappa}\,\eps)$ for any $0<\eps\leq\pi/(2\sqrt{\kappa})-\diam(X)$, see~\cite{Oht07}. A detailed treatment of CAT$(\kappa)$ spaces can be found, for example, in \cite{BriHae99}.

We recall next a notion of weak convergence in metric spaces. Let $(x_n)$ be a bounded sequence in a metric space $(X,d)$. Then
$$r((x_n)):=\inf\left\{\limsup_{n\to\infty} d(y,x_n):y\in X\right\}$$
is called the {\it asymptotic radius} of $(x_n)$ and the {\it asymptotic center} of $(x_n)$ is the set
$$A((x_n)):=\left\{x\in X:\limsup_{n\to\infty} d(x,x_n)=r((x_n))\right\}.$$
An element of $A((x_n))$ will also be referred to as an asymptotic center. If $X$ is a complete $p$-uniformly convex space, then any bounded sequence has a unique asymptotic center since $X$ admits a modulus of uniform convexity that does not depend on the radius of balls (see~\cite{EspFerPia10}). We say that a sequence $(x_n)$ {\it $\Delta$-converges} to $x\in X$ iff $x$ is the unique asymptotic center of every subsequence of $(x_n)$. In this case we call $x$ the {\it $\Delta$-limit} of $(x_n)$. This concept was introduced by Lim~\cite{Lim76}. It is easy to see that in a complete $p$-uniformly convex space, any bounded sequence has a $\Delta$-convergent subsequence.

In the setting of $\CAT(0)$ spaces, $\Delta$-convergence is equivalent to another concept of weak convergence that uses projections on geodesic segments (see \cite{EspFer09}). Reasoning as in \cite{EspFer09}, one can see that this equivalence also holds for $p$-uniformly convex spaces.

Let $(X,d)$ be a metric space. Given $C$ a nonempty subset of $X$, the {\it distance of a point} $x \in X$ to $C$ is $\text{dist}(x,C) := \inf\{d(x,c) : c \in C\}.$ If $B$ and $C$ are nonempty subsets of $X$, we denote $\text{dist}(B,C) := \inf\{d(b,c) : b \in B, c \in C\}$. The {\it metric projection} $P_C$ onto $C$ is the mapping defined by $P_C(x):=\{ c \in C : d(x,c)=\mbox{dist}(x,C)\}$, for every $x\in X$.

Let $X$ be a compete CAT$(\kappa)$ space (with $\diam(X)<\pi/(2\sqrt{\kappa})$ for $\kappa > 0$) and $C$ a closed and convex subset of $X$. Then $P_C : X \to C$ is well-defined and singlevalued (see \cite[Proposition 2.4, page 176]{BriHae99} and \cite[Proposition 3.5]{EspFer09}). If $\kappa = 0$, $P_C$ is also nonexpansive. For $\kappa > 0$ we have that for every $x\in X, y\in C$, $d(P_C(x),y)\leq d(x,y)$. Moreover, in this case $P_C$ is a uniformly continuous mapping because $X$ is bounded and admits a modulus of uniform convexity that does not depend on the radius of balls (see the reasoning in the proof of \cite[Proposition 4.5]{EspNic13}). In fact,  $P_C$ is a Lipschitz mapping (one can easily derive this from the proof of \cite[Theorem 4.1]{Pia13}).

We will also use the following inequality which originates from the work of Reshetnyak (see, for instance, \cite[Theorem 2.3.1]{Jos97} or \cite[ Lemma 2.1]{LanPavSch00} for a simple proof). Let $X$ be a $\CAT(0)$ space. Then for every $x,y,u,v \in X,$
\begin{equation}\label{eq:CAT(0).four.points}
d(x,y)^2 + d(u,v)^2 \le d(x,v)^2 + d(y,u)^2 + 2d(x,u)d(y,v).
\end{equation}

We finish this subsection with two simple results that will be used later on.

\begin{lemma} \label{lemma-conv-seq}
Let $(a_n)$ and $(\eps_n)$ be two sequences of nonnegative numbers. If for every $n \in \mathbb{N}$
\[a_{n+1} \leq a_n + \eps_n,\]
and $\displaystyle \sum_{n = 0}^\infty \eps_n < \infty$, then the sequence $(a_n)$ is convergent.
\end{lemma}

\begin{lemma}\label{lemma:Real.numbers.inequality}
Let $l\ge 0$ and $M>0$. Then there exists a constant $R$ depending only on $l$ and $M$ such that for all $ 0 \le a\leq M$ and $0\leq b\leq 1$,
$$
(a+b)^l\leq a^l+R\,b.
$$
\end{lemma}
\begin{proof}
When $l \in [0,2]$, clearly $(a+b)^l - a^l \le (1 + 2M)b$. For $l>2$, the result is an immediate consequence of a Kaczmarz-Steinhaus inequality~\cite[page~144]{Bul98} which states that there exists a constant $\alpha$ depending only on $l$ such that for all $x\in\R$,
$$
\abs{1+x}^l\leq1+l\,x+\sum_{i=2}^{[l]}\binom{l}{i} x^i+\alpha\abs{x}^l.
$$
\qed
\end{proof}
%
 \subsection{Firmly Nonexpansive Type Mappings in Geodesic  Spaces}
%
Firmly nonexpansive mappings were first introduced by Browder~\cite{Bro67}, under the name of firmly contractive, in the setting of Hilbert spaces, and later by Bruck \cite{Bru73} in the context of Banach spaces. Recently Bruck's definition was extended to a nonlinear setting in~\cite{AriLeuLop14} (see also \cite{Nic13}). We recall this definition using the framework of $p$-uniformly convex spaces.

\begin{definition}
Let $C$ be a nonempty subset of a $p$-uniformly convex space $(X,d)$. We say that a mapping $T:C\to X$ is {\it firmly nonexpansive} iff
\begin{equation}\label{def-lambda-fe}
d(Tx,Ty)\leq d((1-\lambda)x\oplus\lambda Tx,(1-\lambda)y\oplus \lambda Ty),
\end{equation}
for all $x,y\in C$ and $\lambda\in [0,1[$.
\end{definition}

As in the linear case, any sunny nonexpansive retraction is firmly nonexpansive. In particular, the metric projection onto closed and convex subsets of a complete CAT$(0)$ space is firmly nonexpansive. Having $X$ a complete CAT$(0)$ space and a function $f : X \to ]-\infty,+\infty]$ that is convex, lower semi-continuous and proper (i.e., not constantly equal to $+\infty$), the {\it resolvent of $f$} defined by
\[J_\lambda^f(x) := \argmin_{z \in X}\left[f(z) + \frac{1}{2\lambda}d(x,z)^2\right],\]
where $\lambda > 0$, is also firmly nonexpansive. In the subsequent section we will use the following inequality (for its proof see, for instance, \cite[Lemma 3.2]{Bac13}): for every $x,y \in X$,
\begin{equation} \label{ineq-res}
f(J_\lambda^f(x)) + \frac{1}{2\lambda}d(J_\lambda^f(x),x)^2 \le f(y) + \frac{1}{2\lambda}d(x,y)^2 -\frac{1}{2\lambda}d(J_\lambda^f(x),y)^2 .
\end{equation}

Bruck extended in \cite{Bru82} the notion of strongly nonexpansive mappings and introduced the so-called strongly quasi-nonexpansive ones with the aim of proving convergence of iterations built by applying finitely many such mappings. We consider in the following a related property which is useful when studying the asymptotic behavior of sequences obtained by applying a finite number of firmly nonexpansive mappings in the setting of $p$-uniformly convex spaces (see also the notion of strongly attracting mappings in \cite{BauBor96}).

\begin{definition}
Let $C$ be a nonempty  subset of a metric space $(X,d)$. A mapping $T:C\to X$ is said to satisfy {\it property~\eqref{eq:def.P1.maps.beta}} iff $\fix(T)\neq\emptyset$ and there exist $l>0$ and $\beta>0$ such that
\begin{equation*}\tag{$P_1$}\label{eq:def.P1.maps.beta}
d(Tx,u)^l\leq d(x,u)^l-\beta\,d(Tx,x)^l,
\end{equation*}
for all $x\in C$ and $u\in\fix(T)$.
\end{definition}

If $X$ is $p$-uniformly convex with parameter $c$, then every firmly nonexpansive mapping $T:C\to X$ with $\fix(T)\neq\emptyset$ satisfies~\eqref{eq:def.P1.maps.beta} with $l=p$ and $\beta=c/2$. Indeed, given $x\in C$ and $u\in\fix(T)$, for any $\lambda\in[0,1[$ we have that
\begin{align*}
d(u,Tx)^p &=d(Tu,Tx)^p\leq d(u,(1-\lambda)x\oplus\lambda Tx)^p\\
&\leq(1-\lambda)\,d(u,x)^p+\lambda\,d(u,Tx)^p-\frac{c}{2}\,\lambda(1-\lambda)\,d(x,Tx)^p.
\end{align*}
Hence, $d(u,Tx)^p\leq d(u,x)^p-(c/2)\lambda \,d(x,Tx)^p$. Taking limit as $\lambda\to1^-$, we get the conclusion.

A second example of mappings satisfying~\eqref{eq:def.P1.maps.beta} is the metric projection in spaces of curvature bounded above. Let $X$ be a complete $\CAT(\kappa)$ (with $\text{diam}(X) < \pi/(2\sqrt{\kappa})$ for $\kappa>0$) and $C\subseteq X$ closed and convex. Recall that if $\kappa > 0$, then the metric projection $P_C:X\to C$ need not be nonexpansive. However, it satisfies~\eqref{eq:def.P1.maps.beta} with $l=2$ and $\beta=c/2$. To see this, let $x\in X$, $u\in C$ and $\lambda\in ]0,1[$. Then,
\begin{align*}
d(P_C(x),u)^2 &=d(P_C\big((1-\lambda)x\oplus\lambda P_C(x)\big),u)^2 \leq d\big((1-\lambda)x\oplus\lambda P_C(x),u\big)^2   \\
&\leq (1-\lambda)d(x,u)^2+\lambda d(P_C(x),u)^2-\frac{c}{2}\lambda(1-\lambda)d(x,P_C(x)).
\end{align*}
Hence, $d(P_C(x),u)^2\leq d(x,u)^2-(c/2)\lambda\,d(x,P_C(x))$ and we only need to let $\lambda\to1^-$.\\

Given finitely many mappings satisfying $(P_1)$ one can relate the intersection of their fixed point sets with the fixed point set of their composition.

\begin{proposition} \label{prop-int-fp-set}
Let $(X,d)$ be a metric space and let $\{T_i:1\leq i\leq r\}$ be a family of mappings with property~\eqref{eq:def.P1.maps.beta}. If $\bigcap_{i=1}^r\fix(T_i)\neq\emptyset$, then
$$\bigcap_{i=1}^r\fix(T_i)=\fix(T_1\comp\cdots\comp T_r).$$
\end{proposition}
\begin{proof}
Clearly $\bigcap_{i=1}^r\fix(T_i)\subseteq\fix(T_1\comp\cdots\comp T_r)$. Conversely, let $x\in\fix(T_1\comp\cdots\comp T_r)$. If $u\in \cap_{i=1}^r\fix(T_i)$ and  $y=(T_2\comp\cdots\comp T_r)(x)$, then
\begin{equation}\label{eq:prop.Fix.comp}
d(T_1 y,u)\leq d(y,u)\leq\cdots\leq d(x,u).
\end{equation}
Notice that $T_1y=x$. Then, by~\eqref{eq:prop.Fix.comp}, $d(T_1y,u)=d(y,u)$. Since $T_1$ has property~\eqref{eq:def.P1.maps.beta}, we deduce that $\beta\,d(y,T_1y)^l\leq d(u,y)^l-d(u,T_1y)^l=0$, that is, $y=T_1y$. Hence, $x\in \fix(T_2\comp\cdots\comp T_r)$ and  $T_1 x=x$ . Continuing in this way we obtain  the result.
\qed
\end{proof}

The following also purely metric condition is more particular than~\eqref{eq:def.P1.maps.beta} and is still satisfied by any firmly nonexpansive mapping in CAT$(0)$ spaces. In this setting every mapping for which the property below holds is nonexpansive. Note also that in Hilbert spaces, this notion coincides with firm nonexpansivity (see also the remark in \cite[page 658]{Ban14}).

\begin{definition}
Let $C$ be a nonempty subset of a metric space $(X,d)$. A mapping $T:C\to X$ satisfies {\it property~\eqref{eq:def.P2.maps}} iff
\begin{equation*}\tag{$P_2$}\label{eq:def.P2.maps}
2d(Tx,Ty)^2\leq d(x,Ty)^2+d(y,Tx)^2-d(x,Tx)^2-d(y,Ty)^2,
\end{equation*}
for all $x,y\in C$.
\end{definition}

Any mapping that satisfies~\eqref{eq:def.P2.maps} and has fixed points also satisfies~\eqref{eq:def.P1.maps.beta} with $l=2$ and $\beta=1$. Indeed, if $T:C\to X$ satisfies~\eqref{eq:def.P2.maps} and $\fix(T)\neq\emptyset$, then for all $x\in C$ and $u\in\fix(T)$ we have that
$$
2d(Tx,u)^2=2d(Tx,Tu)^2\leq d(x,u)^2+d(u,Tx)^2-d(x,Tx)^2,
$$
from where $d(Tx,u)^2\leq d(x,u)^2-d(x,Tx)^2$. However, the converse implication does not hold. To see this, it suffices to take $T : [-1/4,1/3] \to \mathbb{R}$, $Tx=2x^2$. Then $T$ satisfies~\eqref{eq:def.P1.maps.beta} with $l=2$ and $\beta = 1/3$, but~\eqref{eq:def.P2.maps} fails for $x=-1/4$ and $y=0$.

Note that the metric projection and the resolvent are well-defined in less regular frameworks than complete spaces of curvature bounded above, such as complete uniform Busemann nonpositively curved spaces (see \cite{Jos97}). However, in this setting they may fail to be nonexpansive or satisfy property~\eqref{eq:def.P1.maps.beta}. Since we derive our results concerning the convex feasibility problem and the minimization problem (\ref{def-fct-min}) from the asymptotic behavior of mappings satisfying properties~\eqref{eq:def.P1.maps.beta}  or~\eqref{eq:def.P2.maps}, we state these results in the context of spaces of curvature bounded above.

%
\section{Asymptotic Regularity} \label{section-as-reg}
%
In this section we analyze the asymptotic behavior (also from a quantitative point of view) of the sequence generated by applying iteratively a finite number of firmly nonexpansive type mappings in a cyclic way. We study the case when the considered mappings have a common fixed point, but focus also on the situation when they satisfy the weaker assumption that the fixed point set of their composition is nonempty.

Let $(X,d)$ be a metric space. A mapping $T : X \to X$ is said to be {\it asymptotically regular at $x \in X$} iff $\displaystyle\lim_{n \to \infty}d(T^n x, T^{n+1}x) = 0$ and it is {\it asymptotically regular} iff it is asymptotically regular at every $x \in X$. Likewise, a sequence $(x_n) \subseteq X$ is {\it asymptotically regular} iff $\displaystyle\lim_{n \to \infty} d(x_n, x_{n+1}) =0$. A rate of convergence of $\left(d(x_n, x_{n+1})\right)$ towards $0$ will be called a {\it rate of asymptotic regularity}. As mentioned in the introductory section, asymptotic regularity is an important tool used in the study of the asymptotic behavior of sequences.

In the sequel, having $n,r\in\N$, $r \ge 1$, we use the notation $\m{n}=n\,(\mbox{mod }r)+1$. Let $r \in \N$, $r \ge 1$ and $T_1,\ldots, T_r:X\to X$. We consider the following cyclic method allowing errors in the evaluation of the values of the mappings $T_i$: given $x_0\in X$, define the sequence $(x_n)$ in $X$ by
\begin{equation*}\label{it:inexact}\tag{$A_1$}
d(x_{n+1},T_{\m{n}}x_n)\leq\eps_n,\quad\mbox{for each }n\in\N,
\end{equation*}
where $\displaystyle\sum_{n= 0}^\infty \eps_n<\infty$.

We prove below that the sequence $(x_n)$ is asymptotically regular provided the mappings used to define it satisfy property~\eqref{eq:def.P1.maps.beta} and have a common fixed point. The result applies, in particular, to firmly nonexpansive mappings in the setting of $p$-uniformly convex spaces.

\begin{theorem} \label{thm-as-reg-P1}
Let $(X,d)$ be a metric space and $T_1,\ldots,T_r:X\to X$ satisfy property~\eqref{eq:def.P1.maps.beta}. If $\bigcap_{i=1}^r\fix(T_i)\neq\emptyset$, then given any starting point $x_0 \in X$, the sequence $(x_n)$ defined by \eqref{it:inexact} is asymptotically regular.
\end{theorem}
\begin{proof}
Let $u\in\bigcap_{i=1}^r\fix(T_i)$ and $m \in \N$ such that $\varepsilon_n \le 1$ for $n \ge m$. Because each $T_i$ satisfies \eqref{eq:def.P1.maps.beta} we get that for all $n\geq m + 1$,
\begin{align*}
\beta\,d(x_n,T_{\m{n}}x_n)^l & \leq d(x_n,u)^l-d(T_{\m{n}}x_n,u)^l\\
&\leq \big(d(x_n,T_{\m{n-1}}x_{n-1})+d(T_{\m{n-1}}x_{n-1},u)\big)^l-d(T_{\m{n}}x_n,u)^l\\
&\leq R\,\eps_{n-1}+d(T_{\m{n-1}}x_{n-1},u)^l-d(T_{\m{n}}x_n,u)^l\,\,\, \text{by Lemma }\ref{lemma:Real.numbers.inequality},
\end{align*}
where $R$ is a constant depending only on $l$ and $\displaystyle M \geq d(x_0,u) + \sum_{n= 0}^\infty \eps_n.$ Thus, for all $n \ge m+1$,
\begin{equation*}
\beta\sum_{k=m + 1}^n d(x_k,T_{\m{k}}x_k)^l\leq 
 d(T_{m}x_0,u)^l + R\sum_{k=0}^{\infty}\eps_k.
\end{equation*}
Since $\displaystyle \sum_{k=0}^\infty \eps_k<\infty$, we deduce that $\displaystyle \sum_{k=m+1}^\infty d(x_k,T_{\m{k}}x_k)^l<\infty$, so $\displaystyle\lim_{n \to \infty} d(x_n,T_{\m{n}}x_n)=0$. This finally yields that $(x_n)$ is asymptotically regular.\qed
\end{proof}

It is easy to see that Theorem \ref{thm-as-reg-P1} can be slightly generalized by relaxing property~\eqref{eq:def.P1.maps.beta} imposed on the mappings $T_1,\ldots,T_r$ to the condition that they satisfy the inequality used to define~\eqref{eq:def.P1.maps.beta} only for one common fixed point instead of all fixed points. However, we will not focus on such extensions because the final goal of this work is to apply the above result when the mappings are metric projections, which satisfy property~\eqref{eq:def.P1.maps.beta} in spaces of curvature bounded above.

\begin{corollary}
Let $(X,d)$ be a metric space and $T_1,\ldots,T_r:X\to X$ satisfy property~\eqref{eq:def.P1.maps.beta}. If $\bigcap_{i=1}^r\fix(T_i)\neq\emptyset,$ then the Picard iteration of $T=T_1\comp\cdots\comp T_r$ is asymptotically regular.
\end{corollary}
\begin{proof}
Note that, for $x_0 \in X$ and $k \in \mathbb{N}$, $T^k x_0 = x_{kr}$, where $(x_n)$ is defined by (\ref{it:inexact}) with $\varepsilon_n = 0$ for every $n \in \N$. Hence, for all $k\in\N$,
$$
d(T^k x_0,T^{k+1} x_0) = d(x_{kr},x_{(k+1)r}) \le d(x_{kr},x_{kr+1}) + \ldots + d(x_{kr+r-1},x_{kr+r}),
$$
which implies, by Theorem \ref{thm-as-reg-P1}, that $T$ is asymptotically regular.\qed
\end{proof}

As a consequence, we have that the sequence obtained by the cyclic projection method for a finite number of sets is asymptotically regular in spaces of curvature bounded above.

\begin{corollary} \label{cor-proj-as-ref}
Let $\kappa \in \mathbb{R}$. Suppose $X$ is a CAT$(\kappa)$ space (with ${\rm diam}(X) < \pi/(2\sqrt{\kappa})$ for $\kappa > 0$). Let $C_1, \ldots, C_r \subseteq  X$ be nonempty closed and convex sets such that $\bigcap_{i=1}^r C_i \ne \emptyset$. Given any starting point $x_0 \in X$, let $(x_n)$ be defined by (\ref{it:inexact}) with $T_i = P_{C_i}$, for $i = 1, \ldots, r.$ Then the sequence $(x_n)$ and the mapping $P_{C_1} \comp \ldots \comp P_{C_r}$ are asymptotically regular.
\end{corollary}

We give next a rate of asymptotic regularity for the sequence $(x_n)$ generated by (\ref{it:inexact}) with $\varepsilon_n = 0$ for every $n \in \N$. For simplicity, we prove the theorem for two mappings, but the proof method works equally well in the general case. The result extends \cite[Theorem 5.2]{Nic13} where a quadratic rate of asymptotic regularity in $1/\varepsilon$ was given for the sequence generated by the alternating projection method for two sets in the context of CAT$(0)$ spaces. Since its proof relies on particular properties of the metric projection, the obtained rate of asymptotic regularity is slightly better than the one we give here when restricting to the setting of CAT$(0)$ spaces. However, our result holds not only in a broader framework, but also for more general mappings, namely firmly nonexpansive ones.

\begin{theorem} \label{thm-as-reg-fn-quant}
Let $(X,d)$ be a $p$-uniformly convex space with parameter $c$ and suppose that $T_1,T_2:X\to X$ are firmly nonexpansive mappings with $F:=\fix{T_1}\cap\fix{T_2}\neq\emptyset.$ Consider $x_0\in X$, $u\in F$ and $b>0$ such that $d(x_0,u)\leq b$. Suppose that $(x_n)$ is defined by~\eqref{it:inexact} with $\varepsilon_n = 0$ for every $n \in \N$. Then for all $\eps>0$ and $n\geq\theta(\eps,b,c,p)$, we have that $d(x_n,x_{n+1})\leq\eps$, where
$$
\theta(\eps,b,c,p):=\left\{
\begin{array}{ll}
\displaystyle2\,\left[\frac{2}{c}\left(\frac{4p\,b^p}{c\,\eps^p}\right)^p\right] & \mbox{if }\eps<2b,\\
0 & \mbox{otherwise}.
\end{array}
\right.
$$
Moreover, for all $\eps>0$, $n\geq\tilde\theta(\eps,b,c,p)$ and $i \in \{1,2\}$, we have that $d(x_n,T_i x_n)\leq\eps$, where
$$
\tilde\theta(\eps,b,c,p):=1+\theta(\eps,b,c,p).
$$
\end{theorem}
\begin{proof}
Since for all $n\in\N$, $d(x_n,x_{n+1})\leq d(x_n,u)+d(u,x_{n+1})\leq2b,$ we can assume that $\eps<2b$. Denote
\[\alpha: = \frac{c\,\eps^p}{4p\,b^{p-1}}.\]
Take $N:=\left[2b^p/(c\alpha^p)\right]$. For each $i\in\N$,
\begin{align*}
d(x_{2i+2},u)^p &\leq d(x_{2i+1},u)^p-\frac{c}{2}\,d(x_{2i+1},x_{2i+2})^p  \\
&\leq d(x_{2i},u)^p-\frac{c}{2}\left(d(x_{2i},x_{2i+1})^p + d(x_{2i+1},x_{2i+2})^p\right),
\end{align*}
from where
\begin{equation}\label{eq:proof.rate.RA.0}
\frac{c}{2}\left(d(x_{2i},x_{2i+1})^p + d(x_{2i+1},x_{2i+2})^p\right)\leq d(x_{2i},u)^p-d(x_{2i+2},u)^p.
\end{equation}
Hence, we have that
\begin{equation}\label{eq:proof.rate.RA.1}
\frac{c}{2}\sum_{i=0}^N\left(d(x_{2i},x_{2i+1})^p+d(x_{2i+1},x_{2i+2})^p\right) \leq d(x_0,u)^p-d(x_{2N+2},u)^p \leq b^p.
\end{equation}
Assume that for all $i=0,\ldots,N$, $d(x_{2i},x_{2i+1})^p+d(x_{2i+1},x_{2i+2})^p >\alpha^p$.
Then, using~\eqref{eq:proof.rate.RA.1}, we deduce that $N+1<2b^p/(c\alpha^p),$ a contradiction. This means that there exists $i\in\{0,\ldots,N\}$ such that $d(x_{2i},x_{2i+1})^p+d(x_{2i+1},x_{2i+2})^p\leq\alpha^p.$ Hence,
$d(x_{2i},x_{2i+1})\leq\alpha$ and $d(x_{2i+1},x_{2i+2})\leq\alpha,$ from where $d(x_{2i},x_{2i+2})\leq 2\alpha$. This implies that for all $j\geq i,$ $d(x_{2j},x_{2j+2})\leq 2\alpha.$ By~\eqref{eq:proof.rate.RA.0} and the mean value theorem, we have that for all $j\geq i,$
\begin{align*}
\frac{c}{2}\left(d(x_{2j},x_{2j+1})^p+d(x_{2j+1},x_{2j+2})^p\right)
&\leq d(x_{2j},u)^p-d(x_{2j+2},u)^p \\
&\leq p\,b^{p-1}\,\left(d(x_{2j},u)-d(x_{2j+2},u)\right) \\
&\leq p\,b^{p-1}\,d(x_{2j},x_{2j+2})\leq p\,b^{p-1}\,2\alpha=\frac{c}{2}\eps^p,
\end{align*}
that is, $d(x_{2j},x_{2j+1})^p+d(x_{2j+1},x_{2j+2})^p\leq\eps^p$. Hence, for all $j\geq i$, $d(x_{2j},x_{2j+1})\leq\eps$ and $d(x_{2j+1},x_{2j+2})\leq\eps$. Thus, for all $n\geq 2N$, $d(x_n,x_{n+1})\leq\eps.$

It is easy to see that for every $n \ge 1$ and $i \in \{1,2\}$, 
\[d(x_n,T_i x_n) \le \max\{d(x_n,x_{n+1}),d(x_{n-1},x_n)\}.\] 
Thus, $\tilde\theta$ is a rate of convergence for $(d(x_n,T_i x_n))$ towards $0$.\qed
\end{proof}

\begin{remark}
The above result also holds if we consider a finite family of firmly nonexpansive mappings $T_1,\cdots,T_r$ with $\bigcap_{i=1}^r \fix{T_i} \neq\emptyset.$ In this case, a rate of asymptotic regularity for $(x_n)$, defined by~\eqref{it:inexact} with $\varepsilon_n = 0$ for every $n \in \N$, is
$$
\theta(\eps,b,c,p,r):=\left\{
\begin{array}{ll}
\displaystyle r\,\left[\frac{2}{c}\left(\frac{2\,r\,p\,b^p}{c\,\eps^p}\right)^p\right] & \mbox{if }\eps<2b,\\
0 & \mbox{otherwise}.
\end{array}
\right.
$$
Moreover, for every $i \in \{1,\ldots,r\}$, 
\[\tilde\theta(\eps,b,c,p,r) : = \left[ r/2 \right]  + \theta\left(\frac{\eps}{2 \lceil r/2 \rceil-1},b,c,p\right)\] 
is a rate of convergence for $(d(x_n,T_i x_n))$ towards $0$.
\end{remark}

We prove in the following an asymptotic regularity result for two mappings that satisfy \eqref{eq:def.P2.maps} and for which the fixed point set of the composition is nonempty. To simplify the writing, we consider the sequences $(x_n)$ and $(y_n)$ defined by
\begin{equation*}\label{it:two}\tag{$A_2$}
d(y_n,T_1 x_n) \le \varepsilon_n \quad \text{and} \quad d(x_{n+1},T_2 y_n) \le \delta_n,\quad\text{for each }n\in \N,
\end{equation*}
where $\displaystyle \sum_{n = 0}^\infty\varepsilon_n < \infty$ and $\displaystyle\sum_{n =0}^\infty\delta_n < \infty$.

\begin{theorem} \label{thm-as-reg-P2}
Let $(X,d)$ be a CAT$(0)$ space and let $T_1,T_2:X\to X$ satisfy property~\eqref{eq:def.P2.maps}. If $\fix(T_2\comp T_1)\neq\emptyset$, then given any $x_0 \in X$, the sequences $(x_n)$ and $(y_n)$ defined by~\eqref{it:two} are asymptotically regular.
\end{theorem}
\begin{proof}
Denote $S=T_2\comp T_1$. Then, for every $n \in \mathbb{N}$,
\begin{equation}\label{thm-as-reg-P2-eq1}
d(x_{n+1},Sx_n)  \leq d(x_{n+1},T_2 y_n)+d(T_2 y_n,T_2(T_1x_{n})) \le \delta_n + \varepsilon_n.
\end{equation}
Hence,
\begin{equation}\label{thm-as-reg-eq2}
\lim_{n \to \infty}d(x_{n+1},Sx_{n})=0.
\end{equation}
We also have that for every $n, k \in \mathbb{N}$,
\begin{align*}
& d(Sx_{n+1+k},Sx_{n+1}) \leq d(T_1x_{n+1+k},T_1x_{n+1}) \\
&\quad \leq d(T_1x_{n+1+k},T_1Sx_{n+k}) +d(T_1Sx_{n+k},T_1Sx_{n}) + d(T_1Sx_{n},T_1x_{n+1} )\\
&\quad \leq d(x_{n+1+k},Sx_{n+k}) + d(Sx_{n+k},Sx_{n})+d(Sx_{n},x_{n+1}) \\
& \quad \leq \delta_{n+k}+\eps_{n+k} +\delta_{n}+\eps_{n} + d(Sx_{n+k},Sx_{n}) \quad \text{by } (\ref{thm-as-reg-P2-eq1}).
\end{align*}
This yields, by Lemma \ref{lemma-conv-seq}, that $\left(d(Sx_{n+k},Sx_{n})\right)_{n}$ converges to some $\xi_k \ge 0$.  Furthermore, we have that $\displaystyle \lim_{n \to \infty}d(T_1Sx_{n+k},T_1Sx_{n})=\xi_k$.

Suppose $\xi_1>0$. We prove by induction that $\xi_k=k\,\xi_1$ for all $k\geq1$. For $k=1$ this is clear. Suppose the claim holds true for all $i=1,\ldots,k$. We show that it also holds for $k+1$. Since $T_1$ satisfies \eqref{eq:def.P2.maps}, we have that
\begin{align*}
2d(T_1Sx_{n+k},T_1Sx_{n})^2 &\leq d(Sx_{n+k},T_1Sx_{n})^2+d(Sx_{n},T_1Sx_{n+k})^2 \\
& \quad - d(Sx_{n+k},T_1Sx_{n+k})^2 - d(Sx_{n},T_1Sx_{n})^2.
\end{align*}
Likewise,
\begin{align*}
2d(Sx_{n+k},Sx_{n})^2 & =2d(T_2(T_1x_{n+k}),T_2(T_1x_{n}))^2 \\
& \leq d(T_1x_{n+k},Sx_n)^2+d(T_1x_n,Sx_{n+k})^2 \\
&\qquad - d(T_1x_{n+k},Sx_{n+k})^2 - d(T_1x_n,Sx_n)^2.
\end{align*}
Adding the two above inequalities and using (\ref{eq:CAT(0).four.points}) we obtain that
\begin{align*}
& 2d(T_1Sx_{n+k},T_1Sx_n)^2 + 2d(Sx_{n+k},Sx_n)^2 \\
& \quad \le 2d(Sx_n,Sx_{n+k})d(T_1Sx_n,T_1x_{n+k}) +2d(Sx_n,Sx_{n+k})d(T_1x_n,T_1Sx_{n+k}).
\end{align*}
Dividing by $2d(Sx_{n},Sx_{n+k})$, using the nonexpansivity of $T_1$, the triangle inequality and the notation $\gamma_n^k=d(T_1Sx_{n},T_1Sx_{n+k})^2/d(Sx_n,Sx_{n+k})$,
\begin{align*}
& \gamma_n^k + d(Sx_{n+k},Sx_n) \le d(Sx_n,x_{n+k}) + d(x_n,Sx_{n+k}) \\
& \quad \le d(Sx_n,Sx_{n+k-1})+d(Sx_{n+k-1},x_{n+k}) + d(x_n,Sx_{n-1})+d(Sx_{n-1},Sx_{n+k}).
\end{align*}
Thus,
\begin{align*}
& \gamma_n^k + d(Sx_{n+k},Sx_n) - d(Sx_n,Sx_{n+k-1}) - d(x_n,Sx_{n-1}) - d(Sx_{n+k-1},x_{n+k})\\
& \quad \le d(Sx_{n-1},Sx_{n+k}) \le d(Sx_{n-1},Sx_{n}) +  d(Sx_{n},Sx_{n+k}).
\end{align*}
Letting $n \to \infty$ above and using the induction hypothesis and (\ref{thm-as-reg-eq2}), we get that
$$\xi_{k+1} = \lim_{n \to \infty} d(Sx_n,Sx_{n+k+1}) = (k+1)\xi_1,$$
which finishes the induction. Thus, $\left(Sx_{n}\right)$ is unbounded. Let $u\in\fix(T_2\comp T_1)$. Since, by (\ref{thm-as-reg-P2-eq1}), for all $n \in \mathbb{N}$, 
\[d(Sx_{n+1},u) \leq d(x_{n+1},u) \leq d(x_{n+1},Sx_n)+d(Sx_n,u) \leq \delta_n +\eps_n +d(Sx_{n},u),\]
it follows that $\displaystyle d(Sx_{n+1},u) \le \sum_{n=0}^\infty \delta_n + \sum_{n=0}^\infty \eps_n + d(x_0,u)$, which shows that $\left(Sx_n\right)$ is bounded, a contradiction.

Therefore, $\xi_1=0$, that is, $\displaystyle \lim_{n \to \infty}d(Sx_{n+1},Sx_n) = 0$ which yields, using~\eqref{thm-as-reg-eq2}, that $\displaystyle \lim_{n \to \infty}d(x_{n+1},x_n) = 0$ because
\[d(x_{n+1},x_n)\leq d(x_{n+1},Sx_n)+d(Sx_n,Sx_{n-1})+d(Sx_{n-1},x_{n}).\]
At the same time,
\begin{align*}
d(y_{n+1},y_n) & \leq d(y_{n+1},T_1 x_{n+1}) + d(T_1 x_{n+1}, T_1 x_n) + d(T_1 x_n, y_n)\\
& \le \eps_{n+1} + \eps_n + d(x_{n+1},x_n),
\end{align*}
from where $\displaystyle \lim_{n \to \infty}d(y_{n+1},y_n) = 0$.\qed
\end{proof}

\begin{remark} \label{rmk-as-reg-P2-bound}
Note that in the proof of Theorem \ref{thm-as-reg-P2} one can replace the condition that $\text{Fix}(T_2 \comp T_1) \ne \emptyset$ by the assumption that one of the sequences $(x_n)$ or $(y_n)$ is bounded (and so the other sequence will be bounded too). We show in the next section (see Theorem \ref{thm-conv-P2}) that these two conditions are in fact equivalent.
\end{remark}

We apply in the sequel the above result to the study of a minimization problem which was considered in Hilbert spaces in \cite{AckPre80, BauComRei05} and was recently taken up in \cite{Ban14} in the setting of CAT$(0)$ spaces. Let $(X,d)$ be a CAT$(0)$ space and $f,g : X \to ]-\infty,+\infty]$ convex, lower semi-continuous and proper. Let $\Phi : X \times X \to ]-\infty,+\infty]$,
\[\Phi(x,y) := f(x) + g(y) + \frac{1}{2\lambda}d(x,y)^2,\]
where $\lambda > 0$. We focus on the problem
\begin{equation}\label{def-fct-min}
\argmin_{(x,y)\in X \times X}  \Phi(x,y).
\end{equation}
Denote by $S \subseteq X \times X$ the set of its solutions. One way to approach this problem is to apply alternatively the two resolvents corresponding to the functions $f$ and $g$. In this process one can also allow errors when computing the values of the resolvents. Thus, given a starting point $x_0 \in X$, one can construct the sequences $(x_n)$ and $(y_n)$ defined by (\ref{it:two}), when considering $T_1 = J_\lambda^g$ and $T_2=J_\lambda^f$. Note that if $(x^*,y^*) \in S$, then $y^* = J_\lambda^g(x^*)$ and $x^* = J_\lambda^f(y^*)$, so $\text{Fix}(J_\lambda^f \circ J_\lambda^g) \ne \emptyset$. At the same time, if $x^* \in \text{Fix}(J_\lambda^f \circ J_\lambda^g)$, then $(x^*, J_\lambda^g(x^*)) \in S$. To see this, denote $y^* = J_\lambda^g(x^*)$. Then,  by (\ref{ineq-res}), for each $x \in X$,
\[f(x^*) + \frac{1}{2\lambda}d(x^*,y^*)^2 \le f(x) + \frac{1}{2\lambda}d(x,y^*)^2 - \frac{1}{2\lambda}d(x,x^*)^2\]
and, similarly, for every $y \in X$,
\[g(y^*) + \frac{1}{2\lambda}d(y^*,x^*)^2 \le g(y) + \frac{1}{2\lambda}d(y,x^*)^2 - \frac{1}{2\lambda}d(y,y^*)^2.\]
Adding these two inequalities and applying (\ref{eq:CAT(0).four.points}) we obtain that $(x^*,y^*) \in S$.

Since the resolvents $J_\lambda^f$ and $J_\lambda^g$ are firmly nonexpansive mappings, we can directly apply Theorem \ref{thm-as-reg-P2} to obtain that the sequences $(x_n)$ and $(y_n)$ are asymptotically regular when $S \ne \emptyset$. This property also follows from \cite[Theorem 1]{Ban14}, where the given proof relies on particular properties of the resolvent and uses, instead of the assumption that $S \ne \emptyset$, the weaker condition that $\Phi$ is bounded below. In fact, analyzing this proof, one can immediately extract a rate of asymptotic regularity for the sequences $(x_n)$ and $(y_n)$ and a rate for the computation of $\varepsilon$-solutions of the minimization problem (\ref{def-fct-min}).
\begin{corollary} \label{minpro1}
Let $(X,d)$ be CAT$(0)$ space and suppose that $\Phi$ is bounded below by $m \in \mathbb{R}$. Consider $x_0\in X$ and $b>0$ such that $\Phi\left((J_\lambda^f \circ J_\lambda^g)(x_0),J_\lambda^g(x_0)\right) \leq b$. Suppose that $(x_n)$ and $(y_n)$ are defined by~\eqref{it:two} with $T_1 = J_\lambda^g$, $T_2=J_\lambda^f$, $\varepsilon_n = 0$ and $\delta_n = 0$ for every $n \in \N$. Then for all $\eps>0$ and $n\geq\theta(\eps,b,m,\lambda)$, we have that $$d(y_n,y_{n+1}) \leq d(x_n,x_{n+1}) \leq \eps,$$ where
\begin{equation} \label{cor-min-rate-as}
\theta(\eps,b,m,\lambda):=\left[\frac{2\lambda(b-m)}{\eps^2}\right] + 1.
\end{equation}
Moreover, if $(x^*,y^*) \in S$, then for all $\varepsilon > 0$ and $n \ge 1+\theta\left(\varepsilon \lambda/d(x_0,x^*),b,m,\lambda\right)$,
\[\Phi(x_n,y_n) \le \Phi(x^*,y^*) + \varepsilon.\]
\end{corollary}
\begin{proof}
From the proof of \cite[Theorem 1]{Ban14} we know that for every $x,y \in X$ and $n \in \mathbb{N}$,
\begin{equation}\label{cor-minpro-eq1}
\Phi(x_{n+1},y_n) \le \Phi(x,y) + \frac{1}{2\lambda}\left(d(x_n,x)^2 - d(x_{n+1},x)^2\right).
\end{equation}
Taking above $x:=x_n$ and $y:=y_{n-1}$ we obtain that for $n \ge 1$,
\[d(x_n,x_{n+1})^2 \le 2\lambda\left(\Phi(x_n,y_{n-1}) - \Phi(x_{n+1},y_n)\right).\]
From this, as in the proof of Theorem \ref{thm-as-reg-fn-quant}, we have that $\theta$ is a rate of asymptotic regularity for $(x_n)$.

Suppose $(x^*,y^*) \in S$. Then $(d(x_n,x^*))$ is decreasing and bounded above by $d(x_0,x^*)$. Letting now $x:=x^*$ and $y:=y^*$ in (\ref{cor-minpro-eq1}) and using the fact that, by \cite[Corollary 1]{Ban14}, $\Phi(x_{n+1},y_{n+1}) \le \Phi(x_{n+1},y_n)$ for $n \ge 1$, it follows that for every $\varepsilon > 0$ and $n \ge \theta\left(\varepsilon \lambda/d(x_0,x^*),b,m,\lambda\right)$,
\begin{align*}
\Phi(x_{n+1},y_{n+1}) &\le \Phi(x^*,y^*) + \frac{1}{2\lambda}\left(d(x_n,x^*)^2 - d(x_{n+1},x^*)^2\right)\\
&\le \Phi(x^*,y^*) + \frac{1}{\lambda}d(x_n,x_{n+1})d(x_0,x^*) \le \Phi(x^*,y^*) + \varepsilon.
\end{align*}
\qed
\end{proof}
As already noted in \cite{Ban14} (see also \cite{AckPre80,BauComRei05} for the Hilbert space context), the alternating projection method for two nonempty, closed and convex sets $A$ and $B$ in a complete CAT$(0)$ space $X$ is a particular case of this minimization problem when taking $f = \delta_A$ and $g=\delta_B$, where, for $C\subseteq X$ nonempty, closed and convex,
$$\delta_C : X \to [0, \infty], \quad
\delta_C(x):=\left\{
\begin{array}{ll}
0 & \mbox{if } x \in C,\\
\infty & \mbox{otherwise},
\end{array}
\right.
$$
is the indicator function. Clearly, for every $\lambda > 0$, $J_\lambda^{\delta_C} = P_C$. Thus, the sequences $(x_n)$ and $(y_n)$ obtained as described before are asymptotically regular with a rate quadratic in $1/\varepsilon$ that is given by (\ref{cor-min-rate-as}) with $\lambda=1/2$, $m = 0$ and $b \ge d\left((P_A \circ P_B)(x_0), P_B(x_0)\right)^2$. We point out that, if $A \cap B = \emptyset$, one aims to find best approximation pairs with respect to $A$ and $B$, that is, find $(a,b) \in A \times B$ such that $d(a,b) = \text{dist}(A,B)$. In Hilbert spaces, this problem was studied, for instance, in \cite{BauBor94}.
%
\section{Convergence Results} \label{section-conv}
%
Using the asymptotic regularity results from the previous section we prove in this section  that the sequences generated by (\ref{it:inexact}) and (\ref{it:two}) $\Delta$-converge. Moreover, if we impose the condition that the image of one of the considered mappings is boundedly compact, then we obtain strong convergence. The results that we give below find natural applications to projection methods or the minimization problem defined by~(\ref{def-fct-min}).

\begin{theorem} \label{thm-conv-P1}
Let $(X,d)$ be a complete $p$-uniformly convex space and suppose that $T_1, \ldots, T_r :X\to X$ are firmly nonexpansive mappings. If $F:=\bigcap_{i=1}^r\fix(T_i)\neq\emptyset$, then given any starting point $x_0 \in X$, the sequence $(x_n)$ defined by~\eqref{it:inexact} is $\Delta$-convergent to some $u \in F$. If, in addition, there exists $i \in \{1, \ldots, r\}$ such that $T_i(X)$ is boundedly compact, then $(x_n)$ converges to $u$.
\end{theorem}
\begin{proof}
Let $n \in \mathbb{N}$ and $q \in F$. Because $d(q, x_{n+1})  \le d\left(q, T_{\overline{n}} x_{n}\right) + \varepsilon_{n} \le d(q,x_n) + \varepsilon_n$ we have, by Lemma \ref{lemma-conv-seq}, that the sequence $\left(d(q,x_n)\right)$ converges.

Let $i \in \{1, \ldots, r\}$. For $n \in \mathbb{N}$, let $m_n, j_n \in \mathbb{N}$ with $j_n \in \{0, \ldots, r-1\}$ be such that $n = m_n r + j_n$. Then,
\begin{align*}
d(x_n, T_i x_n) & \le d(x_n, x_{m_n r + i}) +d(x_{m_n r + i}, T_i x_n)\\
& \le  d(x_{m_n r + j_n}, x_{m_n r + i}) + d(x_{m_n r + i}, T_i x_{m_n r + i-1}) + d(T_i x_{m_n r + i-1}, T_i x_n)\\
& \le d(x_{m_n r + j_n}, x_{m_n r + i}) + \varepsilon_{m_n r + i-1} + d(x_{m_n r + i-1}, x_{m_n r + j_n}).
\end{align*}
By Theorem \ref{thm-as-reg-P1}, $(x_n)$ is asymptotically regular, so $\lim_{n \to \infty}d(x_n, T_i x_n)=0$.

Denote by $u$ the unique asymptotic center of $(x_n)$. Let $(x_{n_k}) \subseteq (x_n)$ and assume that its unique asymptotic center is $p$. Then, for every $i \in \{1, \ldots, r\},$
\begin{align*}
d(T_ip, x_{n_k}) \le d(T_ip, T_i x_{n_k}) + d(T_i x_{n_k},x_{n_k}) \le d(p,x_{n_k}) + d(T_i x_{n_k},x_{n_k}).
\end{align*}
Taking limit superior as $k \to \infty$ we conclude that $p \in F$. Then,
\begin{align*}
\limsup_{k \to \infty}d(x_{n_k},p) & \le \limsup_{k \to \infty}d(x_{n_k},u) \le \limsup_{n \to \infty}d(x_n,u) \le \limsup_{n \to \infty}d(x_n,p)\\
& = \lim_{n \to \infty} d(x_n,p) = \lim_{k \to \infty}d(x_{n_k},p),
\end{align*}
which yields, by uniqueness of asymptotic centers, that $p=u$. Thus, $(x_n)$ $\Delta$-converges to $u$.

We may suppose that $T_r(X)$ is boundedly compact. Since $(x_{nr})$ is a bounded sequence in $T_r(X)$ it has a convergent subsequence whose limit will be $u$. Since $\left(d(u,x_n)\right)$ converges, this implies that $(x_n)$ converges to $u$.\qed
\end{proof}

As a straightforward consequence, the sequence generated by the cyclic projection method for a finite number of sets $\Delta$-converges in CAT$(0)$ spaces. Moreover, if one of the sets is locally compact, then we obtain in fact strong convergence. This result extends \cite[Theorem 4.1]{BacSeaSim12} and  a corresponding one mentioned in \cite[page 5]{Bac13} as an application of a proximal splitting algorithm.

\begin{corollary}\label{cor:projections-convergence}
Let $X$ be a complete CAT$(0)$ space and $C_1, \ldots, C_r \subseteq  X$ be nonempty closed and convex sets such that $\bigcap_{i=1}^r C_i \ne \emptyset$. Given any starting point $x_0 \in X$, let $(x_n)$ be defined by (\ref{it:inexact}) with $T_i = P_{C_i}$, for $i = 1, \ldots, r$. Then $(x_n)$ $\Delta$-converges to some point $u \in \bigcap_{i=1}^r C_i \ne \emptyset$. Moreover, if there exists $i \in \{1,\ldots,r\}$ such that $C_i$ is locally compact, then $(x_n)$ converges to $u$.
\end{corollary}

A similar result also holds in spaces of curvature bounded above by some $\kappa > 0$ if assuming an appropriate bound on the diameter and local compactness of one of the sets.

\begin{proposition}
Let $X$ be a complete CAT$(\kappa)$ space with $\kappa > 0$ and for which ${\rm diam}(X) < \pi/(2\sqrt{\kappa})$. Let $C_1, \ldots, C_r \subseteq  X$ be nonempty closed and convex sets such that $\bigcap_{i=1}^r C_i \ne \emptyset$ and there exists $i \in \{1,\ldots,r\}$ such that $C_i$ is locally compact. Given any starting point $x_0 \in X$, let $(x_n)$ be defined by (\ref{it:inexact}) with $T_i = P_{C_i}$, for $i = 1, \ldots, r$. Then $(x_n)$ converges to some point in $\bigcap_{i=1}^r C_i$.
\end{proposition}
\begin{proof}
As in the proof of Theorem \ref{thm-conv-P1}, one has that $\left(d(q,x_n)\right)$ is convergent for every $q \in \bigcap_{i=1}^r C_i$ and $\lim_{n \to \infty}d(x_n, P_i x_n)=0$ for $i = 1, \ldots, r$ (here one applies Corollary~\ref{cor-proj-as-ref} and the fact that $P_i$ is uniformly continuous as mentioned in Subsection~\ref{subsection-geod-sp}).

We can suppose that $C_r$ is locally compact. Since $(x_{nr})$ is a bounded sequence in $C_r$ it has a convergent subsequence $(x_{n_k r})$. Suppose its limit is $u \in X$. Taking into account that, for every $i \in \{1, \ldots, r\}$, $\lim_{k \to \infty}d(x_{n_k r},P_i x_{n_k r}) = 0$, it follows that $u \in \bigcap_{i=1}^r C_i$. Thus, because $\left(d(u,x_n)\right)$ converges, we have that $\lim_{n \to \infty} x_n = u$ and we are done.\qed
\end{proof}

We focus next on the convergence of the sequences defined by (\ref{it:two}) when the considered mappings satisfy $(P_2)$.

\begin{theorem} \label{thm-conv-P2}
Let $(X,d)$ be a complete CAT$(0)$ space. Suppose $T_1,T_2:X\to X$ satisfy property $(P_2)$. Given any starting point $x_0 \in X$, consider the sequences $(x_n)$ and $(y_n)$ defined by (\ref{it:two}). The following statements hold true:
\begin{itemize}
\item[(i)] if ${\rm Fix}(T_2 \circ T_1) = \emptyset$, then $(x_n)$ and $(y_n)$ are unbounded;
\item[(ii)] if ${\rm Fix}(T_2 \circ T_1) \ne \emptyset$, then there exists $u \in {\rm Fix}(T_2 \circ T_1)$ such that $(x_n)$ and $(y_n)$ $\Delta$-converge to $u$ and $T_1 u$, respectively. If, in addition, the image of $T_1$ or $T_2$ is boundedly compact, then the convergence is strong.
\end{itemize}
\end{theorem}
\begin{proof}
(i) Suppose $(x_n)$ is bounded. Denote $S = T_2 \circ T_1$. Then, for every $n \in \mathbb{N}$ and $z \in X$,
\begin{equation}\label{thm-conv-P2-eq1}
\begin{split}
d(x_{n+1}, Sz) & \le d(x_{n+1}, T_2 y_n) + d(T_2 y_n, T_2(T_1z)) \le \delta_n + d(y_n,T_1 z) \\
& \le \delta_n + d(y_n, T_1 x_n) + d(T_1 x_n, T_1 z) \le \delta_n + \varepsilon_n + d(x_n, z).
\end{split}
\end{equation}
Thus, $\lim_{n \to \infty}d(x_{n+1}, Sx_{n}) = 0$. Denote by $u$ the unique asymptotic center of $(Sx_{n})$. Let $(Sx_{n_k}) \subseteq (Sx_{n})$ and assume that its unique asymptotic center is $p$. Then,
\begin{align*}
d(Sp, Sx_{n_k}) & \le d(Sp, Sx_{n_k + 1}) + d(Sx_{n_k + 1},Sx_{n_k}) \\
&\le d(p,Sx_{n_k}) + d(Sx_{n_k},x_{n_k + 1}) + d(x_{n_k + 1},x_{n_k}).
\end{align*}
Taking limit superior as $k \to \infty$ and using Theorem \ref{thm-as-reg-P2} and Remark \ref{rmk-as-reg-P2-bound} we conclude that $p \in \text{Fix}(S)$, which shows that $\text{Fix}(S) \ne \emptyset$.

(ii) If $\text{Fix}(T_2 \circ T_1) \ne \emptyset$, clearly $(x_n)$ and $(y_n)$ are bounded and we continue the reasoning from (i). Applying (\ref{thm-conv-P2-eq1}) with $z = p$ we have, by Lemma \ref{lemma-conv-seq}, that $\left(d(x_{n},p)\right)$ converges. This yields that $\left(d(Sx_{n},p)\right)$ is convergent too. Then,
\begin{align*}
\lim_{k \to \infty}d(Sx_{n_k},p)
& \le \limsup_{k \to \infty}d(Sx_{n_k},u) \le \limsup_{n \to \infty}d(Sx_{n},u) \\
& \le \lim_{n \to \infty} d(Sx_{n},p) = \lim_{k \to \infty}d(Sx_{n_k},p),
\end{align*}
which means that $p = u$. Thus, $(Sx_{n})$ $\Delta$-converges to $u$, from where $(x_{n})$ $\Delta$-converges to $u$ too.

In a similar way one obtains that there exists $v \in  {\rm Fix}(T_1 \circ T_2)$ such that $(y_n)$ $\Delta$-converges to $v$. We need to prove that $v=T_1u$. Since 
\[d(T_2 v,x_{n+1}) \le d(T_2 v,T_2 y_{n}) + \delta_{n} \le d(v,y_{n}) + \delta_{n},\] 
we have that
\[\lim_{n \to \infty} d(u,x_n) \le \lim_{n \to \infty} d(T_2 v,x_{n+1}) \le \lim_{n \to \infty} d(v,y_{n}) \le \lim_{n \to \infty} d(T_1 u,y_{n}).\]
One can also show that $\lim_{n \to \infty} d(T_1 u,y_{n}) \le \lim_{n \to \infty}d(u,x_{n})$, which implies that $v=T_1 u$.

Strong convergence when the image of $T_1$ or $T_2$ is boundedly compact can be obtained as in the proof of Theorem \ref{thm-conv-P1}.\qed
\end{proof}

It follows immediately that we can recover the convergence result for problem (\ref{def-fct-min}) that was proved in \cite[Theorem 1]{Ban14}.
\begin{corollary} \label{minpro2}
Let $(X,d)$ be a complete CAT$(0)$ space and $f,g : X \to ]-\infty,\infty]$ convex, lower semi-continuous and proper. Suppose (\ref{def-fct-min}) has a solution. Given $x_0 \in X$, consider the sequences $(x_n)$ and $(y_n)$ defined by (\ref{it:two}), where $T_1 = J_\lambda^g$ and $T_2=J_\lambda^f$. Then there exists $u \in X$ such that $(u,J_\lambda^g u)$ is a solution of (\ref{def-fct-min}) for which $(x_n)$ and $(y_n)$ $\Delta$-converge to $u$ and $J_\lambda^g u$, respectively.
\end{corollary}

Consequently, one obtains the convergence of the alternating projection method in CAT$(0)$ spaces.

\begin{corollary}
Let $(X,d)$ be a complete CAT$(0)$ space and $A,B \subseteq X$ nonempty, closed and convex. Suppose that $S =  \left\{(a,b) \in A \times B : d(a,b)={\rm dist}(A,B)\right\} \ne \emptyset$. Given $x_0 \in X$, consider the sequences $(x_n)$ and $(y_n)$ defined by (\ref{it:two}), when $T_1=P_B$ and $T_2=P_A$. Then there exists $(a,b) \in S$ such that $(x_n)$ and $(y_n)$ $\Delta$-converge to $a$ and $b$, respectively. If, in addition, $A$ or $B$ are locally compact, then the convergence is strong.
\end{corollary}

Note that in the above result $S \ne \emptyset$ if, for example, one of the sets $A$ or $B$ is bounded.

\begin{example}
We finish this paper by illustrating the behavior of the alternating projection method in the Poincar\'{e} upper half-plane model, see~\cite{BriHae99}. Using \textit{Mathematica$^\circledR$} version 9.0, we  developed an elementary     program,  \url{http://cfne.url.ph/},  where one can observe the behavior of  the Picard iteration for the composition of  projections onto two geodesics.
   Using this program we obtain  Figure~\ref{fig:1} which clearly depicts   the following possible alternatives: the  first one,  Figure~\ref{fig:1}~(a), corresponds to the situation   where the intersection of the two geodesics is nonempty and illustrates strong convergence of the Picard iteration to a point in this intersection, see Corollay~\ref{cor:projections-convergence};  when  the intersection is empty, as in Figure~\ref{fig:1}~(b) and~(c),   the Picard iteration either converges to a best approximation pair (bounded case) or diverges (unbounded case), as is expected from  Theorem~\ref{thm-conv-P2}.
\begin{figure}[h!]
\begin{center}
\includegraphics[scale=.32]{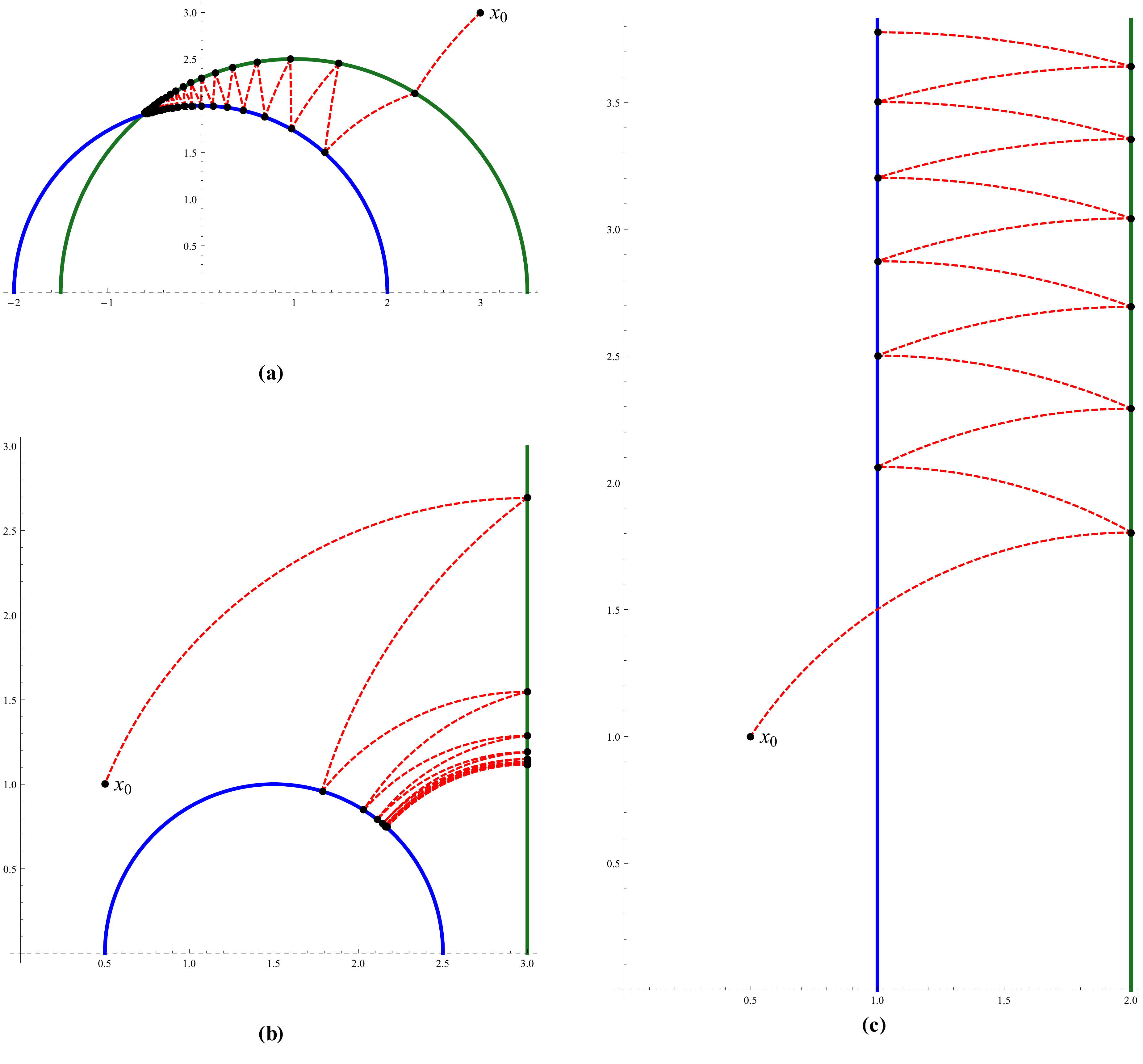}
\caption{Three possible cases for the convergence of alternating projections.}
\label{fig:1}
\end{center}
\end{figure}
\end{example}

\section{Perspectives}
In this paper we have primarily focused on providing an appropriate framework and a general method for the study of the convex feasibility problem for a finite number of sets in a nonlinear setting. For this purpose, it was enough to compose finitely many firmly nonexpansive mappings in a cyclic way. A question which arises naturally is whether one could take infinitely many firmly nonexpansive mappings. Note that, in this case, Theorem \ref{thm-conv-P1} is no longer true without additional assumptions (it is enough to consider the proximal point algorithm in Hilbert spaces where the condition imposed on the stepsize sequence is essential, see \cite{Gul91}).

Regarding the minimization problem (\ref{def-fct-min}), we have given an explicit rate of asymptotic regularity for the sequences $(x_n)$ and $(y_n)$ defined by (\ref{it:two}) with $T_1 = J_\lambda^g$ and $T_2=J_\lambda^f$. It would be interesting to give a quantitative version of Theorem \ref{thm-as-reg-P2} for mappings that satisfy property~\eqref{eq:def.P2.maps} and from where the asymptotic regularity of the sequences $(x_n)$ and $(y_n)$ can be deduced. 

\section{Conclusions}
We have obtained new and more general convergence results to approximate a common fixed point of a finite number of firmly nonexpansive mappings in geodesic spaces. Our results, which are both qualitative and quantitative, can be applied to the cyclic projection method for finitely many sets in geodesic spaces of curvature bounded above or to an abstract minimization problem which allows the study of the alternating projection method for the inconsistent convex feasibility problem for two sets.

\section{Acknowledgements}
David Ariza and Genaro L\'opez were supported by DGES (Grant MTM2012-34847-C02-01), Junta de Andaluc\'{\i}a (Grant P08-FQM-03453). Adriana Nicolae was supported by a grant of the Romanian Ministry of Education, CNCS - UEFISCDI, project number PN-II-RU-PD-2012-3-0152.
Part of this work was carried out while Adriana Nicolae was visiting the University of Seville. She would like to thank the Department of Mathematical Analysis and the Institute of Mathematics of the University of Seville (IMUS) for the hospitality. 
%
%

%
\end{document}